\documentclass[leqno,12pt]{article}

\usepackage{amsmath,amssymb,amsthm,bm,mathrsfs}
\usepackage{enumitem}

\usepackage[a4paper,top=2cm,bottom=2.5cm,left=1cm,right=1cm]{geometry}

\newcommand{\R}{\mathbb{R}}
\newcommand{\N}{\mathbb{N}}

\newcommand{\F}{\mathbb{F}}
\newcommand{\E}{\mathcal{E}}
\newcommand{\OO}{\mathcal{O}}
\renewcommand{\P}{\mathbb{P}}
\newcommand{\divv}{\mathrm{div}}

\newcommand\atopp[2]{\textstyle\genfrac{}{}{0pt}{}{#1}{#2}} 

\DeclareMathOperator*{\esup}{\mathrm{ess\,sup}}
\DeclareMathOperator*{\supp}{\mathrm{supp}}

\newtheorem{theorem}{Theorem}[section]
\newtheorem{lemma}{Lemma}[section]

\theoremstyle{definition}
\newtheorem{remark}{Remark}[section]

\makeatletter

\@addtoreset{equation}{section}
\makeatother

\newcommand{\dd}{{\rm\,d}}

\usepackage{framed,color}
\definecolor{shadecolor}{rgb}{0.9,0.9,0.9}

\title{Solutions of the Navier--Stokes equations with forced rapid space-time decay}

\author{Lorenzo Brandolese\thanks{Université Lyon 1, Institut Camille Jordan, CNRS UMR~5208, 69622 Villeurbanne, France.}
\and
Matthieu Pageard\footnotemark[1]}

\date{\today}

\begin{document}

\maketitle

\abstract{We study the pointwise decay properties of solutions to the incompressible Navier--Stokes equations, both in the space and time variables. It is well known that generic global solutions on $\R^n$ do not decay faster at infinity than $|x|^{-(n+1)}$ and $t^{-(n+1)/2}$ in the pointwise sense. In this paper, we address the control problem of constructing an external forcing and a  solution to the Navier--Stokes equations whose space-time decay properties go beyond these limiting rates. 
A distinctive feature of the forcing term is that its spatial profile can be fixed once and for all, independently of the initial data of the problem, and localized in an arbitrarily small region of $\R^n$. Only the temporal profile of the external force displays a dependency on the initial datum.}



\section{Introduction}

Let $n\geq 2$. We consider the incompressible Navier--Stokes equations on $\R^n$:
\begin{equation}
\label{NS}
\tag{NS}
\begin{cases}
\partial_t u -\Delta u + (u\cdot\nabla)u + \nabla\pi = \F, \\
\divv\,u=0, \\
u(\cdot,0)=u_0.
\end{cases}
\end{equation}
Since the work of Dobrokhotov and Shafarevich \cite{DS94}, it is well known that solutions of the Navier--Stokes equations \eqref{NS} on the whole space $\R^n$ usually decay 
at the spatial infinity at relatively slow decay rates: typically not faster than 
$\mathcal{O}(|x|^{-(n+1)})$ as $|x|\to+\infty$, no matter how fast the initial velocity field~$u_0$ decays, when there is no external force, and even at slower rates, typically like $\mathcal{O}(|x|^{-n})$,
when the external force $\F\not\equiv0$ is not of divergence form. 
This instantaneous spatial spreading effect was sharply measured, \emph{e.g.}, in \cite{BraM02,BraV,KukR,Top25} in the case of zero external forces. The analysis made in the books
\cite{Lem16,Cha25} encompasses the case of non-zero forces.
Motivated by~\cite[Theorem~3.6.2]{Cha25} and \cite{BraOka}, in this paper we will focus on the case
of an external force of divergence form, 
\[
\F=\divv f,
\]
where $f=(f_{i,j})_{1\le i,j\le n}$ is an $(n\times n)$-matrix. 

Concerning the large-time behavior, global solutions usually do not decay faster than $\mathcal{O}(t^{-(n+1)/2})$, in the pointwise sense (\emph{i.e.}, in the $L^\infty(\R^n)$-norm)
and therefore not faster than $\mathcal{O}(t^{-(n+2)/4})$ in the $L^2(\R^n)$-norm, as $t\to+\infty$.
This fact was first discussed in~\cite{Sch91}, then further studied, \emph{e.g.}, in~\cite{BaeJin, MS01, GW1, GW2}.

Only few exceptions to this generic behaviors were found, usually putting
symmetries on the initial data \cite{Bra04i},
 or imposing the initial data to belong 
to specific manifolds of finite codimension, in spaces of strongly localized vorticities, see \cite{GW1,GW2}. 

For example, in the 2D case, one can construct a solution with radial vorticity such that, for all $t>0$, $u(\cdot,t)\in\mathcal{S}(\R^2)$ and the $L^2(\R^2)$-norm decays to zero exponentially fast. But such non-zero solutions invariant under rotations do not exist in 3D, because of topological obstructions. These flows are somehow ``trivial'', in the sense that the nonlinearity identically vanishes and the solution of the Navier--Stokes equations boils down to the solution of the heat equation.

Examples of flows with nontrivial nonlinearity and fast space-time decay are proposed in~\cite{Bra04i}, by putting a dihedral symmetry (in 2D) or a polyhedral symmetry on $\omega_0=\nabla\times u_0$ (in 3D).
In these situations, the symmetries force the 
vorticity $\omega:=\nabla\times u(\cdot,t)$ to have a large number of vanishing moments,
not only when $t=0$ but also for all $t>0$; then one can deduce from the Biot-Savart law that the velocity field decays both as $|x|\to+\infty$ and $t\to+\infty$ at
faster rates, depending on how many vanishing moments of the vorticity persist during the evolution.
For example, in 3D, for flows with icosahedral symmetries, one can achieve
decay rates as $\mathcal{O}(|x|^{-8})$  as $|x|\to+\infty$, 
and $\mathcal{O}(t^{-4})$ as $t\to+\infty$ in the pointwise sense,
or as $\mathcal{O}(t^{-5/2})$ in the $L^2(\R^3)$-norm. These decays are twice faster than for generic flows.
No example of 3D flow decaying faster than $|x|^{-8}$ is known so far.  
The invariant manifold approach of~\cite{GW1,GW2} would be effective to prove the existence
of flows with arbitrary large algebraic \emph{time} decay rates, but the obtained solutions do not have, in general, better spatial decays than generic solutions. 
The recent papers~\cite{McOTop23,Top25} provide a deep understanding of the spatial behavior of generic solutions, by exploring the insight given by asymptotic spaces introduced by the same authors \cite{McOTop11}, but the analysis therein does not cover solutions with over-critical space decay.

\medskip
    
In this work, we will consider well localized initial data that are generic, in the sense that we require no specific symmetry structure. 
We will address the control problem consisting in constructing a forcing term of divergence form $\F=\divv f$, where $f=f(x,t)$ is a $(n\times n)$-matrix,
with the goal of obtaining arbitrary fast decay to solutions of~\eqref{NS}, both in the space and time variables.
Remarkably, the spatial profile of the forcing matrix $f$ can be arbitrarily described a priori and does not depend on the initial datum: 
in other words, $f$ will be of the form
\[
f(x,t)=\sum_{\atopp{\alpha\in\N^n}{0\le|\alpha|\le m-1}} \chi_\alpha(x)A_\alpha(t),
\]
where the functions $\chi_\alpha$ can be chosen once and for all independently of the initial datum~$u_0$; only a finite number
of time-depending parameters (the coefficients of $A_\alpha(t)$) have to be tuned
after the initial datum to achieve the desired fast decay condition.
The number of the required tuning parameters depends on the rate that one wants to achieve.
As it will be clear from the proof, the force and the corresponding solution
will be constructed algorithmically from the datum $u_0$.

\subsection{Main result}

In order to state our results, we need to introduce the following conditions on the initial datum $u_0$.

Let $T>0$, $m\in\N$ and $\kappa>0$. When dealing with local-in-time solutions defined on an interval $[0,T]$, we will consider initial data which are 
measurable functions satisfying a decay condition of the form
\begin{subequations}
\begin{equation}
\label{small:sta}
\bigl(\sqrt{T}+T^{(m+3)/2}\bigr)\esup_{x\in \R^n}\,(1+|x|)^{n+1+m}|u_0(x)| \le \kappa.
\end{equation}
When dealing with global-in-time solutions, the conditions on~$u_0$ will be of the form
\begin{equation}
\label{hypo}
\begin{split}
&\esup_{x\in \R^n}\,(1+|x|)^{n+1+m}|u_0(x)|\le\kappa, \\
&\int (1+|x|)^{m+1}|u_0(x)|\dd x\le\kappa.
\end{split}
\end{equation}
In the latter case, it will be useful to prescribe moment conditions on~$u_0$, namely
\begin{equation}
    \label{hypo2}
    \int x^\alpha u_0(x)\dd x=0\qquad\forall\alpha\in\N^n,\,0\le|\alpha|\le m.
\end{equation}
\end{subequations}

Let us now state our main result. 

\begin{theorem}
\label{th:1}
Let $n\ge2$ and $m\ge0$ be two integers. Let $T\in(0,+\infty]$
and $u_0$ be a divergence-free vector field.
Let $K=\prod_{i=1}^n[a_i,b_i]$ be a box in $\R^n$, with $a_i<b_i$ for all $i=1,\ldots,n$.
There exist a function $\chi\in C^\infty_0(\R^n;\R)$ with $\supp\chi\subset K$, 
a constant $\kappa=\kappa(n,m,\chi)>0$ and a force $f_m$ given by $f_0\equiv0$ and
\begin{equation}
       \label{control}
       f_m(x,t)=\sum_{\atopp{\alpha\in\N^n}{0\le|\alpha|\le m-1}}  A_\alpha(t)\chi_\alpha(x)\qquad\forall(x,t)\in\R^n\times[0,T)\qquad(m\ge1),
\end{equation}
where $\chi_\alpha(x):=\frac{(-1)^{|\alpha|}}{\alpha!}\partial^\alpha\chi(x)$ and $A_\alpha(t)$ is an $(n\times n)$-matrix depending only on $\alpha$, $t$ and $u_0$, such that the following properties hold true:
\begin{enumerate}[label=\roman*)]
    \item If $T<+\infty$ and $u_0$ satisfies the smallness condition \eqref{small:sta}, then there exists a solution $u$ of the Navier--Stokes equations \eqref{NS} with $\F=\divv f_m$, verifying the decay property
\begin{equation}
       \label{conc}
       \esup_{t\in[0,T]}|u(x,t)| = \OO(|x|^{-(n+1+m)})\qquad\text{as $|x|\to+\infty$}, 
\end{equation}
If additionally $u_0(x)=o(|x|^{-(n+1+m)})$ as $|x|\to+\infty$, then there exists a solution $u$ of the Navier--Stokes equations \eqref{NS} with $\F=\divv f_{m+1}$ such that 
\begin{equation}
       \label{conc2}
       \esup_{t\in[0,T]}|u(x,t)| = o(|x|^{-(n+1+m)})\qquad\text{as $|x|\to+\infty$}, 
\end{equation}

    \item If $T=+\infty$ and $u_0$ satisfies the smallness conditions \eqref{hypo} and the moment conditions \eqref{hypo2}, then there exists a global solution $u$ of the Navier--Stokes equations \eqref{NS} with $\F=\divv f_m$, verifying the decay properties 
\begin{equation}
       \label{concl}
       \begin{split}
       &\esup_{t\ge0}|u(x,t)| = \OO(|x|^{-(n+1+m)})\qquad\text{as $|x|\to+\infty$}, \\
       &\esup_{x\in\R^n}|u(x,t)| = \OO(t^{-(n+1+m)/2})\qquad\text{as $t\to+\infty$}.
       \end{split}
\end{equation}
Moreover, the external force $f_m$ satisfies 
\begin{equation}
       \label{sizef}
       \|f_m(t)\|_{L^\infty(\R^n)} = \OO(t^{-(n+3+m)/2})\qquad\text{as $t\to+\infty$}.
\end{equation}
\end{enumerate}
\end{theorem}

The last part of Conclusion~i) is mainly interesting for $m=0$, as it provides the existence of solutions with
the over-critical spatial decay rate $o(|x|^{-n-1})$  as $|x|\to+\infty$, under minimal assumptions on $u_0$, 
using a single matrix-valued compactly supported forcing term $f_0(x,t)=A_0(t)\chi(x)$.
\begin{remark}
The reason behind the slow decay properties (in general not faster than
$\OO(|x|^{-(n+1)})$ as $|x|\to+\infty$ and $\OO(t^{-(n+1)/2})$ as $t\to+\infty$, as mentioned in the introduction) of solutions to the Navier--Stokes equations \eqref{NS} is the slow decay of the kernel $F(\cdot,t)$ of the operator $e^{t\Delta}\P\mbox{div}$: from \cite{Miyakawa 2002 FE}, we have
\[
\begin{cases}
    |F(x,t)| \lesssim |x|^{-(n+1)} \\
    |F(x,t)| \lesssim t^{-(n+1)/2}
\end{cases}
\qquad\forall(x,t)\in(\R^n\backslash\{0\})\times(0,\infty),
\]
and these estimates are known to be optimal.
The key idea for overcoming these limiting decay rates is to construct a control force so as to eliminate all the terms decaying 
slower than $|x|^{-(n+1+m)}$ as $|x|\to+\infty$, for some $m\in\N$, in an appropriate asymptotic expansion of $F(\cdot,t)$. 
This is possible when the forcing term $f$ and the solutions are related
by a somewhat specific relation, which takes the form
\begin{equation}
    \label{controluu}
    f(x,t):=
    \sum_{\atopp{\alpha\in\N^n}{0\le|\alpha|\le m-1}} \chi_\alpha(x)
    \biggl(\int y^\alpha (u\otimes u)(y,t)\dd y\biggr),
\end{equation}
where the functions $(\chi_\alpha)_{\alpha\in\N^n}$
in $C^\infty_0(\R^n)$ satisfy the moment conditions 
\begin{equation}
\label{hyp:chi}
\int x^\beta\chi_\alpha(x)\dd x=\delta_{\alpha,\beta},
\qquad
\forall\alpha,\beta\in\N^n,\;0\le|\beta|\le m,
\end{equation}
where $\delta_{\alpha,\beta}$ is the Kronecker symbol. In this way, by choosing $m$ large, we will be able to construct solutions with arbitrarily fast decay.
Of course, an external force $\F=\divv f$ with $f$ given by \eqref{controluu} cannot be defined after the solution, because the solution itself depends on~$f$. For this reason, the force $f$ and the solution of~\eqref{NS} have to be constructed~\emph{at the same time}.

This leads us to define, for any $m\in\N$, the bilinear operator
\[
\widetilde
B_m(u,v)(t):=\int_0^t e^{(t-s)\Delta}\P\divv
\Big((u\otimes v) - \sum_{\atopp{\alpha\in\N^n}{0\le|\alpha|\le m-1}} \chi_\alpha\int y^\alpha (u\otimes v)(y,\cdot)\dd y\Big)(s)\dd s.
\]
The introduction of $\widetilde B_m$ seems to be a new idea.
When $m=0$, the above summation is vacuous and $\widetilde B_m$ boils down to 
the usual Navier--Stokes bilinear operator, given by
\[
\widetilde B_0(u,v)(t):=\int_0^te^{(t-s)\Delta}\P\divv(u\otimes v)(s)\dd s.
\]
When $m\ge1$, the bilinear operator $\widetilde B_m$
encodes the presence of a forcing term
acting on the flow with the needed structure, 
specified by~\eqref{controluu}.
The solution $u$ and the force~$f$ will be constructed simultaneously by a fixed point argument: \emph{a fortiori} both $f$ and $u$ will only depend on the initial datum $u_0$.
The advantage of this strategy is that the spatial profile of $f$ is fixed once and for all by the 
compactly supported function~$\chi=\chi(x)$, which is arbitrarily chosen
and independent on $u_0$. This means that,
for any time, it is possible to control the spatial decay as $|x|\to+\infty$ of the flow at any rate, by acting on the flow in a bounded region with a suitable (time-dependent) linear combination of an arbitrary control function.

The method is robust and could be applied to construct solutions in
a variety of functional settings, under assumptions on the data of different nature. The framework that we choose, consisting in using
weighted $L^\infty$-spaces, is probably the simplest one, as it readily provides strong pointwise estimates, and all the needed estimates are elementary.
With some more technicalities, one could also express localization conditions in other forms, for example, computing suitable weighted norms of the initial data and of the solution. 

In our global existence result, the smallness assumption~\eqref{hypo} is far from being optimal: one often prefers to express the smallness condition on $u_0$ using a scaling invariant norm (such as the $L^n(\R^n)$-norm, or the weaker Koch-Tataru's norm). 
This is indeed possible: condition~\eqref{hypo} could be replaced
by a scale-invariant smallness condition, together with a decay-moment condition of the form $\esup_{x\in \R^n}\,(1+|x|)^{n+1+m}|u_0(x)|<\infty$ and
$\int (1+|x|)^{m+1}|u_0(x)|\dd x<\infty$.
However, such an improvement on condition~\eqref{hypo} would be at the price of developing additional persistence decay results for solutions, making the proof longer and somewhat less transparent.
 
\end{remark}

\subsection{Preliminaries}

In order to quantify the decay properties stated in Theorem \ref{th:1}, we introduce the following function spaces.
Let $T>0$. We define $X_{m,T}$ to be the space of measurable 
vector fields in~$\R^n\times[0,T]$ such that
\[
\|u\|_{X_{m,T}}:=\esup_{x\in\R^n,\,t\in[0,T]}\,(1+|x|)^{n+1+m}|u(x,t)|<\infty,
\]
and the subspace $Y_{m,T}\subset X_{m,T}$ composed of those functions $u\in X_{m,T}$ satisfying
\[
\lim_{|x|\to+\infty}\esup_{t\in[0,T]}\,(1+|x|)^{n+1+m}|u(x,t)|=0.
\]
Similarly, we define $X_m$ to be the space of measurable vector fields in~$\R^n\times[0,\infty)$ such that
\[
\|u\|_{X_m}:=\esup_{x\in\R^n,\,t\ge0}\,(1+|x|)^{n+1+m}|u(x,t)| + \esup_{x\in\R^n,\,t\ge0}\,(1+t)^{(n+1+m)/2}|u(x,t)| <\infty.
\]
As for the initial datum, we define $E_m$ to be the space of measurable vector fields $a\in\R^n$ such that
\[
\|a\|_{E_m}:=\esup_{x\in\R^n}\,(1+|x|)^{n+1+m}|a(x)|<\infty,
\]
and the subspace $\E_m\subset E_m$ of those functions $a\in E_m$ such that 
\[
\lim_{|x|\to+\infty}(1+|x|)^{n+1+m}|a(x)|=0.
\]

We now recall some decay properties of the involved kernels.
The kernel $F=(F_{j;k,l})_{1\le j,k,l\le n}$ of the operator $e^{t\Delta}\P\divv$ satisfies,
for some constant $C>0$ independent on $(x,t)\in (\R^n\backslash\{0\})\times(0,+\infty)$, and for any $r_1,r_2\ge0$ such that $r_1+r_2=n+1$,
\begin{equation}
    \label{usuF0}
    |F(x,t)|\le C|x|^{-r_1}t^{-r_2/2}\qquad\forall(x,t)\in(\R^n\backslash\{0\})\times(0,\infty).
\end{equation}
Moreover, $F$ satisfies the scaling relation
\[
F(x,t)=t^{-(n+1)/2}\Phi\left(\frac{x}{\sqrt t}\right),
\qquad\forall(x,t)\in\R^n\times(0,\infty),
\]
where $\Phi:=F(\cdot,1)$ is a smooth function on~$\R^n$, satisfying, for any $\alpha\in\N^n$ and some constant $C_\alpha>0$,
the pointwise estimates
\[
\partial^\alpha\Phi(x)\le C_\alpha(1+|x|)^{-(n+1+|\alpha|)},
\qquad\forall x\in\R^n.
\]
See~\cite{Miyakawa 2002 FE}.
In particular, for any $\alpha\in\N^n$ and some $C_\alpha>0$,
\begin{equation}
\label{usuF}
\begin{cases}
    |\partial^\alpha F(x,t)| \le C_\alpha|x|^{-(n+1+|\alpha|)} \\
    |\partial^\alpha F(x,t)| \le C_\alpha t^{-(n+1+|\alpha|)/2}
\end{cases}
\qquad\forall(x,t)\in(\R^n\backslash\{0\})\times(0,\infty).
\end{equation}
We also deduce that $\Phi\in(L^1\cap L^\infty)(\R^n)$, and 
\begin{equation}
    \label{normF}
    \|F(\cdot,t)\|_1 = \|\Phi\|_1t^{-1/2},\qquad\forall t>0.
\end{equation}

The heat kernel is denoted $g_t(x):=(4\pi t)^{-n/2}\exp(-|x|^2/4t)$.
It satisfies, for any $\alpha\in\N^n$ and some $c_\alpha>0$,
\begin{equation}
\label{usugt}
\begin{cases}
    |\partial^\alpha g_t(x)| \le c_\alpha|x|^{-(n+|\alpha|)} \\
    |\partial^\alpha g_t(x)| \le c_\alpha t^{-(n+|\alpha|)/2}
\end{cases}
\qquad\forall(x,t)\in(\R^n\backslash\{0\})\times(0,\infty),
\end{equation}
In particular, for any $r_1,r_2\ge0$ such that $r_1+r_2=n$,
\begin{equation}
    \label{usugt0}
    g_t(x)\le C|x|^{-r_1}t^{-r_2/2}\qquad\forall(x,t)\in(\R^n\backslash\{0\})\times(0,\infty),
\end{equation}
for some constant $C>0$ depending only on the space dimension.

\subsection*{Notation} 

We denote by $e^{t\Delta}$ the heat semigroup. The space $L^2_\sigma(\R^n)$ denotes the restriction of $L^2(\R^n)$ to divergence-free vector fields, and the orthogonal projection from $L^2(\R^n)$ onto $L^2_\sigma(\R^n)$ is denoted by $\P$. 
For any $x\in\R^n$ and $\alpha\in\N^n$, we denote $x^\alpha:=x_1^{\alpha_1}\ldots x_n^{\alpha_n}$, $\partial^\alpha:=\partial_{x_1}^{\alpha_1}\ldots\partial_{x_n}^{\alpha_n}$ and $|\alpha|:=\alpha_1+\ldots+\alpha_n$. 
We denote $\int = \int_{\R^n} $.
For two vectors $a,b\in\R^n$, we denote by $a\otimes b$ the $(n\times n)$-matrix given by $(a\otimes b)_{i,j}=a_ib_j$. For an  $(n\times n)$-matrix $A=A(x)$, we define its divergence by $(\mathrm{div}\,A)_i:=\sum_j\partial_jA_{j,i}$.
The notation $A(t)\lesssim B(t)$ means that there exists a constant~$c>0$ independent of~$t$ such that $A(t)\le cB(t)$.

\section{Proof of Theorem \ref{th:1}}

Let $(\chi_\alpha)_{\alpha\in\N^n}$ be the family of smooth functions compactly supported in the box $K=\prod_{i=1}^n[a_i,b_i]\subset\R^n$ and satisfying the moment conditions~\eqref{hyp:chi}, constructed in Appendix \ref{construction}. We introduce, for any $m\in\N$, the matrix-valued function 
$R_m(u,v)=(R_m(u,v)_{k,l})_{1\le k,l\le n}$, defined by $R_0(u,v):=u\otimes v$ and
\[
R_m(u,v)(x,t)
:=(u\otimes v)(x,t)-\sum_{\atopp{\alpha\in\N^n}{0\le|\alpha|\le m-1}} \chi_\alpha(x)\int y^\alpha (u\otimes v)(y,t)\dd y\qquad(m\ge1).
\]
First of all, let us observe that, for any $u,v$ in $X_{m,T}$ or $X_m$, and any $0\le|\alpha|\le m$, the integral $\int y^\alpha(u\otimes v)(y,s)\dd y$ converges for almost every $s\in[0,T]$, so $R_m(u,v)$ and $R_{m+1}(u,v)$ are well defined.
Let us now define, for any $m\in\N$, the bilinear operator 
\begin{equation}
    \label{bilin:new}
    \widetilde B_m(u,v)(x,t)
    :=\int_0^t\!\!\!\int F(x-y,t-s):R_m(u,v)(y,s)\dd y\dd s, 
\end{equation}
where the column indicates a double summation over the subscripts $k,l=1,\ldots,n$: 
\[
\widetilde B_m(u,v)_j(x,t)
=\int_0^t\!\!\!\int \sum_{k,l=1}^n F_{j;k,l}(x-y,t-s) R_m(u,v)_{k,l}(y,s)\dd y\dd s\qquad(j=1,\ldots,n).
\]
In the subsequent estimates, all the components will be treated in the same way, so it will be more convenient to use the formulation \eqref{bilin:new}.

\medskip

Let us first prove the following estimates for the heat semigroup.

\begin{lemma} 
    \label{lem:heat}
    Let $T>0$. The operator $e^{t\Delta}$ maps continuously $E_m$ into $X_{m,T}$, and there exists a constant $C_0=C_0(n,m)>0$ such that, for any $a\in E_m$,
    \begin{equation}
        \label{est1:heat}
        \|e^{t\Delta}a\|_{X_{m,T}} \le C_0\big(1+T^{(m+1)/2}\big)\|a\|_{E_m}.
    \end{equation}
    Moreover, if $a\in\E_m$, then $e^{t\Delta}a\in Y_{m,T}$.
    
    If conditions \eqref{hypo}-\eqref{hypo2} hold true for some $\kappa>0$, with $a$ instead of $u_0$, then 
    \begin{equation}
        \label{est2:heat}
        \|e^{t\Delta}a\|_{X_m} \le C_0\left(\|a\|_{E_m} + \int |y|^{m+1}|a(y)|\dd y\right).
    \end{equation}
\end{lemma}

\begin{proof}
For any $x\in\R^n$, 
\[
e^{t\Delta}a(x) = (g_t\ast a)(x) = \int g_t(x-y)a(y)\dd y,
\]
where $g_t(x)=(4\pi t)^{-n/2}\exp(-|x|^2/4t)$ is the heat kernel. Thus,
\[
|e^{t\Delta}a(x)|\le\|a\|_\infty\le\|a\|_{E_m}.
\]
Now and for the rest of the proof, let $|x|\ge1$. Using \eqref{usugt0}, we estimate 
\[
\begin{split}
|e^{t\Delta}a(x)|
&\le \sup_{|y|\le |x|/2}g_t(x-y) 
\int_{|y|\le |x|/2} |a(y)|\dd y
 + \sup_{|y|\ge|x|/2} |a(y)|\int_{|y|\ge|x|/2}g_t(x-y)\dd y \\
&\lesssim t^{(m+1)/2}|x|^{-(n+1+m)}\|a\|_{E_m} + (1+|x|)^{-(n+1+m)}\|a\|_{E_m} \\
&\lesssim |x|^{-(n+1+m)}\big(1+t^{(m+1)/2}\big)\|a\|_{E_m}.
\end{split}
\]
This implies \eqref{est1:heat} for some constant $C_0=C_0(n,m)>0$.

Suppose now that $a\in\E_m$. Using \eqref{usugt0}, we estimate 
\[
\begin{split}
|e^{t\Delta}a(x)|
&\le \sup_{|y|\le |x|/2}g_t(x-y) 
\int_{|y|\le |x|/2} |a(y)|\dd y
 + \esup_{|y|\ge|x|/2} |a(y)|\int_{|y|\ge|x|/2}g_t(x-y)\dd y \\
&\lesssim t^{(m+2)/2}|x|^{-(n+2+m)}\|a\|_{E_m} + \esup_{|y|\ge|x|/2} |a(y)|,
\end{split}
\]
so
\[
(1+|x|)^{n+1+m}|e^{t\Delta}a(x)| \lesssim t^{(m+2)/2}|x|^{-1}\|a\|_{E_m} + \esup_{|y|\ge|x|/2} (1+|y|)^{n+1+m}|a(y)|,
\]
and $e^{t\Delta}a\in Y_{m,T}$. 

Let us now prove \eqref{est2:heat}. 
The vanishing moment condition \eqref{hypo2} reads
\begin{equation}
    \label{hypo2bis}
    \int y^\alpha a(y)\dd y = 0\qquad\forall\alpha\in\N^n,\,0\le|\alpha|\le m,
\end{equation}
thus we can decompose
\begin{equation}
\label{decomp-etd-Y}
\begin{split}
e^{t\Delta}a(x)
=
&\int_{|y|\le |x|/2} \biggl[g_t(x-y)-\sum_{0\le |\alpha|\le m} \frac{(-1)^{|\alpha|}}{\alpha!}\partial^\alpha_x g_t(x)y^\alpha\biggr]a(y)\dd y
+
\int_{|y|\ge |x|/2} g_t(x-y)a(y)\dd y \\
&-
\sum_{0\le |\alpha|\le m} \frac{(-1)^{|\alpha|}}{\alpha!}\partial^\alpha_xg_t(x)\int_{|y|\ge |x|/2} y^\alpha a(y)\dd y. 
\end{split}
\end{equation}
We first focus on the spatial decay. Let us estimate the second integral. From the pointwise decay of $a$, we immediately obtain
\[
\int_{|y|\ge |x|/2} |g_t(x-y)| \, |a(y)| \dd y \lesssim  (1+|x|)^{-(n+1+m)}\|a\|_{E_m}.
\]
From Taylor's formula, applying condition \eqref{hypo} and the first decay property in \eqref{usugt}, the first integral is bounded by
\[
\sum_{|\alpha|=m+1}
\biggl(\int |y|^{|\alpha|}|a(y)|\dd y\biggr)
\esup_{|y|\ge |x|/2}|\partial_x^\alpha g_t(y)|
\lesssim (1+|x|)^{-(n+1+m)}\int |y|^{m+1}|a(y)|\dd y.
\]
Using \eqref{usugt} again, 
the third integral is bounded by  
\[
\sum_{0\le |\alpha|\le m}
|\partial^\alpha_xg_t(x)|\int_{|y|\ge |x|/2}|y|^{|\alpha|} |a(y)|\dd y
\lesssim (1+|x|)^{-(n+1+m)}\|a\|_{E_m}.
\]
Hence,
\[
(1+|x|)^{n+1+m} |e^{t\Delta}a(x)| \lesssim \|a\|_{E_m} + \int |y|^{m+1}|a(y)|\dd y.
\]
For proving the time decay, we can use again Taylor's formula and the vanishing moment condition \eqref{hypo2bis} to write
\[
e^{t\Delta}a(x)
=(m+1)\sum_{|\alpha|=m+1}\int_0^1(1-\theta)^m\int \frac{(-y)^\alpha}{\alpha!}\partial_x^\alpha g_t(x-\theta y)a(y)\dd y\dd\theta.
\]
Using now the second decay property in \eqref{usugt} as well as \eqref{hypo},
\[
|e^{t\Delta}a(x)|\lesssim t^{-(n+1+m)/2}\int |y|^{m+1}|a(y)|\dd y.
\]
Collecting the above estimates, and taking possibly $C_0$ larger, we finally obtain~\eqref{est2:heat}.
\end{proof}

We now prove the following estimates for the bilinear operator $\widetilde B_m$.

\begin{lemma}
    \label{lem:bilin}
    Let $T>0$. There exists a constant $C_1=C_1(n,m,\chi)>0$ such that the following properties hold true. The operator $\widetilde B_m$ maps continuously $X_{m,T}\times X_{m,T}$ into $X_{m,T}$ and $X_m\times X_m$ into $X_m$, and
    \begin{equation}
\label{est:bilin}
\begin{cases}
    \|\widetilde B_m(u,v)\|_{X_{m,T}} \le C_1(\sqrt{T}+T)\|u\|_{X_{m,T}}\|v\|_{X_{m,T}}, \\
    \|\widetilde B_m(u,v)\|_{X_m} \le C_1\|u\|_{X_m}\|v\|_{X_m}.
\end{cases}
\end{equation}
Moreover, $\widetilde B_{m+1}$ maps continuously $X_{m,T}\times X_{m,T}$ into $Y_{m,T}$, and 
\begin{equation}
\label{est2:bilin}
\|\widetilde B_{m+1}(u,v)\|_{X_{m,T}} \le C_1(\sqrt{T}+T)\|u\|_{X_{m,T}}\|v\|_{X_{m,T}}.
\end{equation}
\end{lemma}

\begin{proof}
For $|x|\le1$, we estimate 
\[
|\widetilde B_m(u,v)(x,t)| \le\int_0^t \|F(t-s)\|_1\|R_m(u,v)(\cdot,s)\|_\infty \dd s,
\]
so that from \eqref{normF}, $|\widetilde B_m(u,v)(x,t)|\lesssim \sqrt{t}\|u\|_{X_{m,T}}\|v\|_{X_{m,T}}$ and $|\widetilde B_m(u,v)(x,t)|\lesssim\|u\|_{X_m}\|v\|_{X_m}$.

Let $|x|\ge1$. The moment conditions~\eqref{hyp:chi} on~$\chi_\alpha$ imply the important cancellation properties
\begin{equation}
    \label{moment-rm}
    \int y^\alpha R_m(u,v)(y,t)\dd y=0\qquad\forall t\ge0,\quad
    \forall\alpha\in\N^n,\;0\le|\alpha|\le m.
\end{equation}
This allows us to write
\[
\begin{split}
\widetilde
B_m(u,v)(x,t)
=
&\int_0^t\int_{|y|\le|x|/2} 
\biggl[F(x-y,t-s)-\sum_{0\le|\alpha|\le m-1}\frac{(-1)^{|\alpha|}}{\alpha!}
\partial^\alpha F(x,t-s)y^\alpha\biggr] : R_m(u,v)(y,s)\dd y\dd s \\
&+
\int_0^t\int_{|y|\ge |x|/2} F(x-y,t-s) : R_m(u,v)(y,s)\dd y\dd s \\
&-
\sum_{0\le|\alpha|\le m-1}\frac{(-1)^{|\alpha|}}{\alpha!}\int_0^t \partial^\alpha F(x,t-s)
\int_{|y|\ge|x|/2} y^\alpha R_m(u,v)(y,s)\dd y\dd s. 
\end{split}
\]
In the first integral, we use Taylor's formula to rewrite the term inside the brackets as
\[
\int_0^1 m(1-\theta)^{m-1}
\sum_{|\alpha|=m}
\frac{(-1)^{|\alpha|}}{\alpha!}\partial^\alpha F(x-\theta y,t-s)y^\alpha\dd \theta.
\]
From this and the first estimate in \eqref{usuF} for the kernel~$F$, we deduce that
\begin{equation}
\label{decomp-bilin}
\begin{split}
|\widetilde B_m(u,v)(x,t)|
\lesssim\; 
&|x|^{-(n+1+m)}\int_0^t\int_{|y|\le |x|/2}|y|^m |R_m(u,v)(y,s)|\dd y\dd s \\
&+
\int_0^t \|F(t-s)\|_1 \esup_{|y|\ge|x|/2}|R_m(u,v)(y,s)|\dd s \\
&+
\sum_{0\le|\alpha|\le m-1}|x|^{-(n+1+|\alpha|)}
\int_0^t\int_{|y|\ge |x|/2}|y|^{|\alpha|}|R_m(u,v)(y,s)|\dd y\dd s.
\end{split}
\end{equation}
Using now the decay \eqref{normF} and the pointwise estimate 
\[
|R_m(u,v)(y,s)| \le (1+|y|)^{-2(n+1+m)}\|u\|_{X_{m,T}}\|v\|_{X_{m,T}},
\]
we finally obtain
\[
\begin{split}
|\widetilde B_m(u,v)(x,t)|
&\lesssim\; 
\Big(t|x|^{-(n+1+m)} + \sqrt{t}(1+|x|)^{-2(n+1+m)} + t|x|^{-1-2(n+1+m)}\Big)\|u\|_{X_{m,T}}\|v\|_{X_{m,T}}  \\
&\lesssim |x|^{-(n+1+m)}\big(\sqrt{t}+t\big)\|u\|_{X_{m,T}}\|v\|_{X_{m,T}}.
\end{split}
\]
This proves the first estimate in \eqref{est:bilin}. 

If we consider the operator $\widetilde B_{m+1}$ instead, owing to the moment conditions \eqref{moment-rm}, we can use Taylor's formula at order $m$, and estimate the first integral on the right-hand side of \eqref{decomp-bilin} by $t|x|^{-(n+2+m)}$. This implies that $\widetilde B_{m+1}(u,v)\in Y_{m,T}$ for all $u,v\in X_{m,T}$, and proves~\eqref{est2:bilin}.

To prove the second estimate in \eqref{est:bilin}, we come back to \eqref{decomp-bilin} and use this time that 
\begin{equation}
\label{est:R}
|R_m(u,v)(y,s)|\le (1+|y|)^{-(n+1+m)}(1+s)^{-(n+1+m)/2}\|u\|_{X_m}\|v\|_{X_m},
\end{equation}
which implies
\[
\begin{split}
|\widetilde B_m(u,v)(x,t)|
&\lesssim 
|x|^{-(n+m+1)}\|u\|_{X_m}\|v\|_{X_m}.\\
\end{split}
\]
To obtain the time decay, we decompose 
\[
\begin{split}
\widetilde
B_m(u,v)(x,t)
=
&\int_0^{t/2}\!\!\!\int 
 \biggl[F(x-y,t-s)-\sum_{0\le|\alpha|\le m-1}\frac{(-1)^\alpha}{\alpha!}
 \partial^\alpha F(x,t-s)y^\alpha\biggr] : R_m(u,v)(y,s)\dd y\dd s \\
&+
\int_{t/2}^t\int F(x-y,t-s) : R_m(u,v)(y,s)\dd y\dd s. 
\end{split}
\]
Using Taylor's formula as before and the second estimate in~\eqref{usuF},
\[
\begin{split}
|\widetilde B_m(u,v)(x,t)|
&\lesssim t^{-(n+1+m)/2}\int_0^{t/2}\!\!\!\int|y|^{m}|R_m(u,v)(y,s)|\dd y\dd s \\
&+\esup_{\tau\ge t/2}\|R_m(u,v)(\cdot,\tau)\|_\infty\int_{t/2}^t \|F(t-s)\|_1\dd s \\
&\lesssim t^{-(n+1+m)/2}\|u\|_{X_m}\|v\|_{X_m}.
\end{split}
\]
We deduce the existence of a constant $C_1=C_1(n,m,\chi)$ such that
\begin{equation}
\label{bilti}
\|\widetilde B(u,v)\|_{X_m}\le C_1\|u\|_{X_m}\|v\|_{X_m}.
\end{equation}
This proves the second estimate in \eqref{est:bilin}.
\end{proof}

We now conclude the proof of Theorem \ref{th:1}. Let us introduce the mapping
\begin{equation}
    \label{psi}
    \Psi_m(u)(t) := e^{t\Delta}u_0 - \widetilde B_m(u,u)(t).
\end{equation}

\medskip

\noindent {\it Proof of assertion i) of Theorem \ref{th:1}.} We perform a fixed point argument. From the estimates of Lemmas \ref{lem:heat} and \ref{lem:bilin},
\[
\|\Psi_m(u)\|_{X_{m,T}} \le C_0\big(1+T^{(m+1)/2}\big)\|u_0\|_{E_m} + C_1(\sqrt{T}+T)\|u\|_{X_{m,T}}^2.
\]
Choosing $\kappa$ sufficiently small in \eqref{small:sta}, which is always possible also in the case of large initial data~$u_0$ (\emph{e.g.} by reducing $T>0$), we have 
\[
4C_0C_1\big(1+T^{(m+1)/2}\big)(\sqrt{T}+T)\|u_0\|_{E_m} < 1.
\]
Hence $\Psi_m$ is a contraction of the closed ball $\big\{u\in X_{m,T} : \|u\|_{X_{m,T}}\le2C_0\big(1+T^{(m+1)/2}\big)\|u_0\|_{E_m}\big\}$. Therefore, $\Psi_m$ admits a fixed point $u\in X_{m,T}$, which solves the equation 
\begin{equation}
    \label{NSm-int}
    u(t) = e^{t\Delta}u_0 - \widetilde B_m(u,u)(t).
\end{equation}
Now let us define
$A_\alpha(t):=\int y^\alpha (u\otimes u)(y,t)\dd y$ and, accordingly with~\eqref{control}, 
\begin{equation}
\label{choi:f}
f_m(x,t) :=
\sum_{\atopp{\alpha\in\N^n}{0\le|\alpha|\le m-1}} A_\alpha(t)\chi_\alpha(x).
\end{equation}
Then, the solution~$u$ to~\eqref{NSm-int} is indeed the solution of  
the Navier--Stokes equations~(NS), with external force $\F=\divv f_m$,
written in their integral form.
This proves the first part of assertion {\it i)} of Theorem \ref{th:1}.

Using now estimate \eqref{est2:bilin} and arguing as before, $\Psi_{m+1}$ admits a fixed point $u\in X_{m,T}$, which satisfies the integral formulation
\[
u(t) = e^{t\Delta}u_0 - \widetilde B_{m+1}(u,u)(t)
\]
of the Navier--Stokes equations \eqref{NS} with $\F=\divv f_{m+1}$. From Lemma \ref{lem:bilin}, we know that $\widetilde B_{m+1}(u,u)\in Y_{m,T}$. Assuming further that $u_0\in\E_m$, we get from Lemma \ref{lem:heat} that $e^{t\Delta}u_0\in Y_{m,T}$. Hence, $u\in Y_{m,T}$. This completes the proof of assertion {\it i)} of Theorem \ref{th:1}.

\medskip

\noindent {\it Proof of assertion ii) of Theorem \ref{th:1}.} Since $u_0$ satisfies conditions \eqref{hypo}-\eqref{hypo2}, Lemmas \ref{lem:heat} and \ref{lem:bilin} yield
\[
\|\Psi_m(u)\|_{X_m} \le C_0\left(\|u_0\|_{E_m} + \int |y|^{m+1}|u_0(y)|\dd y\right) + C_1\|u\|_{X_m}^2.
\]
Choosing $\kappa$ sufficiently small in \eqref{hypo} so that
\[
4C_0C_1\left(\|u_0\|_{E_m} + \int |y|^{m+1}|u_0(y)|\dd y\right) < 1,
\]
$\Psi_m$ admits a fixed point $u\in X_m$, which solves the integral formulation \eqref{NSm-int} of the Navier--Stokes equations \eqref{NS} with $\F=\divv f_m$ and $f_m$ given by \eqref{choi:f}.

To conclude the proof of Theorem \ref{th:1}, it remains to show the decay property \eqref{sizef} for $f_m$. We immediately have 
\begin{equation}
    \label{sizef0}
    \|f_m(t)\|_{L^\infty(\R^n)} \lesssim \sum_{\atopp{\alpha\in\N^n}{0\le|\alpha|\le m-1}} \int |y|^{|\alpha|} |u(y,t)|^2 \dd y.
\end{equation}
For any $\alpha\in\N^n$ with $0\le|\alpha|\le m-1$, let us split the integral as 
\[
\int |y|^{|\alpha|} |u(y,t)|^2 \dd y = \int_{|y|\le\sqrt t} |y|^{|\alpha|} |u(y,t)|^2 \dd y + \int_{|y|\ge\sqrt t} |y|^{|\alpha|} |u(y,t)|^2 \dd y.
\]
In the first integral, we have 
\[
\int_{|y|\le\sqrt t} |y|^{|\alpha|} |u(y,t)|^2 \dd y \le (1+t)^{-(n+1+m)} \|u\|_{X_m}^2 \int_{|y|\le\sqrt t} |y|^{|\alpha|} \dd y \lesssim (1+t)^{-(n+2+2m-|\alpha|)/2}.
\]
For the second integral, we estimate 
\[
\int_{|y|\ge\sqrt t} |y|^{|\alpha|} |u(y,t)|^2 \dd y \le \|u\|_{X_m}^2 \int_{|y|\ge\sqrt t} |y|^{|\alpha|} (1+|y|)^{-2(n+1+m)} \dd y \lesssim (1+t)^{-(n+2+2m-|\alpha|)/2}.
\]
From \eqref{sizef0}, we finally deduce that 
\[
\|f_m(t)\|_{L^\infty(\R^n)}\lesssim (1+t)^{-(n+3+m)/2},
\]
which is the desired decay rate \eqref{sizef} for the external force $f_m$. This completes the proof of Theorem \ref{th:1}. \hfill\qed

\appendix

\section{Construction of the spatial control profile}
\label{construction}

In this appendix, we provide a construction of the family of smooth compactly supported functions $(\chi_\alpha)_{\alpha\in\N^n}$ satisfying the moment relations \eqref{hyp:chi}, used in the proof of Theorem \ref{th:1}.

\medskip

Let $a,b\in\R$, with $a<b$. Let $\psi\in C_0^\infty(\R;\R)$ be a smooth function such that $\supp\psi\subset[a,b]$, $\psi\ge0$ and $\int \psi(s)\dd s=1$.
We want to show that there exists a polynomial 
$P(s):=\sum_{j=0}^m a_js^j$ of degree~$m$ such that the function $\phi(s):=P(s)\psi(s)$ satisfies the moment conditions
\begin{equation}
    \label{cond:phi}
    \int t^k \phi(s) \dd s
    =
    \begin{cases}
    1&\text{if $k=0$}, \\
    0&\text{if $k=1,\ldots,m$.}
    \end{cases}
\end{equation}
For any $\ell\ge0$, let us consider the moments 
\[
\mu_\ell := \int s^\ell \psi(s) \dd s,
\]
as well as the $(m+1)\times(m+1)$ matrix of moments $M=(\mu_{i+j})_{0\le i,j\le m}$.
In this way, conditions \eqref{cond:phi} rewrite as the linear system
\begin{equation}
    \label{linear}
    \sum_{j=0}^m M_{k,j}a_j
    =
    \begin{cases}
    1&\text{if $k=0$}, \\
    0&\text{if $k=1,\ldots,m$.}
    \end{cases}
\end{equation}
For any $a=(a_0,\ldots,a_m)\in\R^{m+1}$,
\[
a^\top Ma = \sum_{i,j=0}^m a_ia_j\mu_{i+j} = \int \biggl(\sum_{i=0}^m a_is^i\biggr)^2 \psi(s) \dd s.
\]
From the assumptions on $\psi$, the right-hand side is positive for every $a\neq0$. Therefore, $M$ is symmetric and positive definite, hence invertible. The linear system \eqref{linear} then admits a unique solution. 

Let us now consider two $n$-uplets $(a_i)_{1\le i\le n}$ and $(b_i)_{1\le i\le n}$ of real numbers such that $a_i<b_i$ for all $i=1,\ldots,n$. For each interval $[a_i,b_i]$, we consider a smooth function $\phi_i\ge0$ such that $\supp\phi_i\subset[a_i,b_i]$ and satisfying conditions \eqref{cond:phi}. We construct the box $K:=\prod_{i=1}^n [a_i,b_i]\subset\R^n$ and the function $\chi\in C_0^\infty(\R^n;\R)$ defined by
\[
\chi(x):=\prod_{i=1}^n\phi_{i}(x_i)\qquad\forall x=(x_1,\ldots,x_n)\in\R^n,
\]
which satisfies $\supp\chi\subset K$ and, for any $\gamma\in\N^n$, the moment conditions
\begin{equation}
    \label{cond:chi}
    \int x^\gamma \chi(x) \dd x
    =
    \begin{cases}
    1&\text{if $\gamma=0$}, \\
    0&\text{if $0<|\gamma|\le m$.}
    \end{cases}
\end{equation}
For $\alpha\in\N^n$, we then define
\[
\chi_\alpha(x) := \textstyle\frac{(-1)^{|\alpha|}}{\alpha!}\partial^\alpha\chi(x).
\]
An integration by parts yields, for any $\alpha,\beta\in\N^n$ with $0\le|\beta|\le m$,
\[
\int x^\beta\chi_\alpha(x)\dd x
=
\begin{cases}
\binom{\beta}{\alpha}
\int x^{\beta-\alpha}\chi(x)&\text{if $\alpha_i\le \beta_i$ for all $i=1,\ldots,n$}, \\
0&\text{otherwise.}
\end{cases}
\]
From the moment conditions \eqref{cond:chi}, we finally deduce that
\[
\int x^\beta\chi_\alpha(x)\dd x
=\delta_{\alpha,\beta}\qquad\forall\,0\le|\beta|\le m,
\]
and by construction, $\supp\chi_\alpha\subset K$ for any $\alpha\in\N^n$.

\small



\begin{thebibliography}{99}




\bibitem{BaeJin}
H.~O.~Bae and B.~J.~Jin, \emph{Temporal and spatial decays for the Navier--Stokes equations},
Proc. Roy. Soc. Edinburgh Sect. A {\bf 135} (2005), 461--478.

\bibitem{BraM02}
L.~Brandolese and Y.~Meyer, \emph{On the instantaneous spreading for the Navier--Stokes system in the whole space}, ESAIM Control Optim. Calc. Var. \textbf{8} (2002), 273--285.


\bibitem{Bra04i}
L.~Brandolese, \emph{Space-time decay of Navier--Stokes flows invariant under rotations}, Math. Ann. \textbf{329} (2004), no.~4, 685--706.

\bibitem{BraOka}
L.~Brandolese and T.~Okabe, \emph{Annihilation of slowly-decaying terms of Navier--Stokes flows by external forcing}, Nonlinearity \textbf{34} (2021), no.~3, 1733--1757.

\bibitem{BraV}
L.~Brandolese and F.~Vigneron, \emph{New asymptotic profiles of nonstationary solutions of the Navier--Stokes system}, J. Math. Pures Appl. (9) \textbf{88} (2007), no.~1, 64--86.




\bibitem{Cha25}
D.~Chamorro, \emph{Introduction aux équations de Navier--Stokes incompressibles}, CNRS editions, EDP science, Savoir actuels (2025),
ISBN-13 978-2759836345.

\bibitem{DS94}
Y.~Dobrokhotov and A.~I.~Shafarevich, \emph{Some integral identities and remarks on the decay at infinity of solutions of the Navier--Stokes equations}, Russian J. Math. Phys. \textbf{2} (1994), 133--135.

\bibitem{FujM}
Y.~Fujigaki and T.~Miyakawa, \emph{Asymptotic profiles of nonstationary incompressible Navier--Stokes flows in the whole space}, SIAM J. Math. Anal. \textbf{33} (2001), 523--544.

\bibitem{Fujita Kato}
H.~Fujita and T.~Kato, \emph{On the Navier--Stokes initial value problem. I}, Arch. Ration. Mech. Anal. \textbf{16} (1964), 269--315.

\bibitem{GW1}
T.~Gallay and C.~E.~Wayne, \emph{Invariant manifolds and the long-time asymptotics of the Navier--Stokes and vorticity equations on $\mathbf{R}^2$}, Arch. Ration. Mech. Anal. \textbf{163} (2002), no.~3, 209--258.

\bibitem{GW2}
T.~Gallay and C.~E.~Wayne, \emph{Long-time asymptotics of the Navier--Stokes and vorticity equations on ${\mathbb R}^3$}, R. Soc. Lond. Philos. Trans. Ser. A Math. Phys. Eng. Sci. \textbf{360} (2002), no.~1799, 2155--2188.

\bibitem{Kato}
T.~Kato, \emph{Strong $L^p$-solutions of the Navier--Stokes equation in ${\bf R}^m$, with applications to weak solutions}, Math. Z. \textbf{187} (1984), 471--480.

\bibitem{KukR} 
I.~Kukavica and E.~Reis, \emph{Asymptotic expansion for solutions of the Navier--Stokes equations with
potential forces}, J. Differential Equations, {\bf 250} (2011), no. 1, 607--622.

\bibitem{Lem}
P.~G.~Lemari\'e-Rieusset, \emph{Recent developments in the Navier--Stokes problem}, Chapman \& Hall/CRC Research Notes in Mathematics, vol.~431, Chapman \& Hall/CRC, Boca Raton, FL, 2002.

\bibitem{Lem16}
   {P.~G.~Lemari\'e-Rieusset},
   \emph{The Navier--Stokes problem in the 21st century},
   {CRC Press, Boca Raton, FL},
   {2016},
   {xxii+718},
   {ISBN 978-1-4665-6621-7}.

\bibitem{McOTop11}
R.~McOwen and P.~Topalov, \emph{Groups of asymptotic diffeomorphisms},
Discrete Contin. Dyn. Syst., {\bf 36} (2016), no. 11, 6331--6377.

\bibitem{McOTop23}
R.~McOwen and P.~Topalov, \emph{Spatial Asymptotic Expansions in the Navier--Stokes Equation},
Int. Math. Res. Not. IMRN {\bf 4} (2024), 3391--3441.

\bibitem{Miyakawa 2002 FE}
T.~Miyakawa, \emph{Notes on space-time decay properties of nonstationary incompressible Navier--Stokes flows in $\mathbf{R}^n$}, Funkcial. Ekvac. \textbf{45} (2002), 271--289.


\bibitem{MS01}
T.~Miyakawa and M.~E.~Schonbek, \emph{On optimal decay rates for weak solutions to the Navier--Stokes equations in $\mathbb R^n$}, in Proceedings of Partial Differential Equations and Applications (Olomouc, 1999), Math. Bohem. \textbf{126} (2001), 443--455.



\bibitem{Sch91}
M.~E.~Schonbek, \emph{Lower bounds of rates of decay for solutions to the Navier--Stokes equations}, J. Amer. Math. Soc. \textbf{4} (1991), no.~3, 423--449.

\bibitem{Top25}
P.~Topalov, \emph{Spatial Decay/Asymptotics in the Navier--Stokes Equation},
Russian J. Math. Phys. \textbf{32} (2025) no.~1, 196--209.




\end{thebibliography}
\end{document}